\documentclass{amsart}

\newtheorem{theorem}{Theorem}[section]
\newtheorem{lemma}[theorem]{Lemma}

\theoremstyle{definition}

\newtheorem{example}[theorem]{Example}

\theoremstyle{remark}
\newtheorem{remark}[theorem]{Remark}

\numberwithin{equation}{section}

\begin{document}

\title {Fekete's Subadditive Lemma revisited.}
\author {L\'{a}szl\'{o}  Tapolcai  Greiner}

\begin {abstract} The subject of this paper is an extension of the  Fekete's Subadditive Lemma for a set of submultiplicative functionals on infinite product of compact spaces. Our method it can be considered as an unfolding of he ideas \cite{1}Theorem 3.1 and our main result is the extension of the symbolic dynamics results of \cite{4}.
\end{abstract}
\maketitle

 \section{Extremal submultiplicative sequences}

\subsection{ Subadditive Lemma.} To begin with we recall the Subadditive Lemma  \cite {1}.  If the sequence $\{ a_n \;\; n=1,2,\dots \}$ of real numbers is subadditive in the sense
$$a_{n+m}\leq a_n+a_m \;\; n,m=1,2,\dots ,$$
then 
$$\lim_{n\rightarrow\infty}\frac{1}{n}a_n=\inf_n\frac{1}{n}a_n$$
where $\inf_n$ means infimum. 

We shall use also a {\em multiplicative form of this lemma.}
\begin{lemma}
  Let the sequence $\{ c_n \;\; n=1,2,\dots \}$ of positive numbers be submultiplicative in the sense
$$c_{n+m}\leq c_nc_m \;\; n,m=1,2,\dots ,$$
then 
$$\lim_{n\rightarrow\infty}c_n^{\frac{1}{n}}=\inf_nc_n^{\frac{1}{n}}$$
where $\inf_n$ means infimum. 
\end{lemma}

\begin {proof}
 the proof follows from the obvious connection that $a_n=\log c_n$ is subadditive if $c_n$  is submultiplicative.
\end{proof}

Let  $\{X_i \;\; i=1,2,\dots\}$ be a sequence of compact spaces and consider the  product space   $\Pi_{i=1}^\infty X_i$.
We shall call {\em flow}  the elements of $\Pi_{i=1}^\infty X_i$ and  {\em word} the elements of a finite product of  $\{X_i \;\; i=1,2,\dots\}$.

More particularly, if $\sigma_i$ is an element of $X_i$, then  $\sigma_i\sigma_{i+1}\dots\sigma_j$ is a word in  $\Pi_{n=i}^j X_n$
and  $\sigma=\sigma_1\sigma_2\dots $ is a flow.

Let $\phi$ be a positive functional on the words, which is {\em submultiplicative} in the following sense. If $\sigma_i\sigma_{i+1}\dots\sigma_j$ is a word and $i<k<j$, then
$$\phi(\sigma_i\sigma_{i+1}\dots\sigma_j)\leq \phi(\sigma_i\sigma_{i+1}\dots\sigma_k)\phi(\sigma_{k+1}\sigma_{k+2}\dots\sigma_j).$$

If $\phi$ is continuous on each $\{X_i \;\; i=1,2,\dots\}$ then there is  $\sigma_1^{[n]}\dots \sigma_n^{[n]}$ for every $n$ that
$$\phi(\sigma_1^{[n]}\dots \sigma_n^{[n]})= \sup\{\phi(\sigma_1\dots\sigma_n): \;\;   \sigma\in\Pi_{i=1}^\infty X_i\}$$
since $\{X_i \;\; i=1,2,\dots\}$ are compact sets. If
\begin{equation}
\Phi_n=\phi(\sigma_1^{[n]}\dots \sigma_n^{[n]})
\end{equation}
then it is easy to verify that the sequence $\{\Phi_n \;\; n=1,2,\dots\}$ of positive numbers  is submultiplicative. Hence, it follows from the Submultiplicative Lemma, that $\{\Phi_n^{\frac{1}{n}} \;\; \;\, n=1,2,\dots\}$ is convergent and the limit is the {\em infimum} of  $\{\Phi_n^{\frac{1}{n}} \;\; \;\, n=1,2,\dots\}$.

 Notice, that the words  $\{ \sigma_1^{[n]}\dots \sigma_n^{[n]}  \;\; \;\; n=1,2,\dots \},$  where $\phi$ reach its maximum in $\{X_n  \;\; n=1,2,\dots \},$  {\em  are not the prefix of a single flow} $\sigma=\sigma_1\sigma_2\dots \;\; \;\; (\sigma_i\in X_i )$.

Let $\Phi_*$ be the limit of  $\{\Phi_n^{\frac{1}{n}} \;\; \;\, n=1,2,\dots\}$. For every flow  $\sigma\in\Pi_{i=1}^\infty X_i$, 
 there is also a limit
$$\phi_*(\sigma)= \lim_{n\rightarrow\infty}\phi(\sigma_1\dots\sigma_n)^{\frac{1}{n}}$$ 
obviously,  $\phi_*(\sigma)\leq \Phi_*$.

{\em Our main result is that there exists $\sigma$ with  $\phi_*(\sigma)=\Phi_*$.}
\begin {theorem}
 There is  $\sigma\in \Pi_{i=1}^\infty X_i$ with  
\begin{equation}
\phi (\sigma_1\sigma_2\dots\sigma_n)\geq\Phi_*^n
\end{equation}
  for $n=1,2,\dots$  i.e. (1.2) holds for every prefix of $\sigma$.
\end{theorem} 

\begin{proof}
 Consider the nested sequence
\begin{equation}
\mathcal {T}_n=\{\sigma_1,\sigma_2\dots\sigma_n:  \;\; \phi(\sigma_1,\sigma_2\dots\sigma_k )\geq\Phi_*^k \;\; \;\; k=1,2,\dots n\}
\end{equation}
of compact sets. If  $\mathcal {T}_n \neq\emptyset$ for every $n$, then by the finite intersection property, 
$$\bigcap_{n=1}^\infty\mathcal {T}_n \neq\emptyset$$

and each $\sigma$ in  $\cap_{n=1}^\infty\{\mathcal {T}_n \;\; n=1,2,\dots\}$ satisfies (1.2) for $n=1,2,\dots$.

Hence, if we prove that each set $\mathcal {T}_n$ is not empty then we have done.\\

Suppose that there is $n$ that $\mathcal {T}_n=\emptyset$ and we get  to a contradiction. If  there is $n$ that $\mathcal {T}_n=\emptyset$ then for every $\sigma$ there is  $k<n$ that
\begin{equation}
 \phi (\sigma_1\sigma_2\dots\sigma_k)<\Phi_*^k.
\end{equation}
and hence
 $$\phi (\sigma_1\sigma_2\dots\sigma_k)^\frac{1}{k}\leq\Phi_*-\varepsilon$$
since $X_i$ is compact and  $\phi$, restricted to $X_i$, is continuous.

 Every subword of $\sigma$ of length $n$ contains index $k$ with (1.4). Hence  for every $\sigma$ and $m>n$
$$\phi (\sigma_1\sigma_2\dots\sigma_m)\leq \phi (\sigma_1\sigma_2\dots\sigma_{k_1})\phi (\sigma_{k_1+1}\sigma_{ k_1+2}
\dots\sigma_{k_2})\dots\phi(\sigma_{k_l+1}\sigma_{k_l+2}\dots\sigma_m) $$
where
$$\phi (\sigma_{k_j+1}\sigma_{ k_j+2}\dots\sigma_{k_{j+1}})<\Phi_*^{k_{j+1}-k_j} \;\; \;\, j=1,\dots  ,l .$$
It follows
$$\phi (\sigma_1\sigma_2\dots\sigma_n)\leq C(\Phi_*-\varepsilon)^{k_l}$$
where $C=\max\{1, \Phi_1^n\}$,   \ $k_i<n$ and  $m-n<k_l<m$.\\
          
  For every $m>n$ consider  the maximal  word $\Phi_m$. I.e  let $(\sigma_1^{[m]}\dots \sigma_m^{[m]}$ be the word
satisfied (1.1). Then
\begin{equation}
\Phi_m^{\frac{1}{m}}\leq C^{\frac{1}{m}}(\Phi_*-\varepsilon)^{\frac{k_l}{m}}
\end{equation}
and (1.5) is valid also  if  $m\rightarrow\infty$. Thus the contradiction \  $\Phi_*\leq\Phi_*-\varepsilon$ \   is obtained and hence
 $\mathcal {T}_n$ is not empty.
\end{proof}

\begin{remark}
 An important case is when  $\{X_i \;\; i=1,2,\dots\}$ are finite sets. In this case we can drop the continuouty conditions, obviously.

That the theorem is not trivial also in this case, shows the next example.

Let 
\begin{equation}
A_0=\left( \begin{array}{cc} 0.3 & 0 \\ 0 & 0.3 \end{array}\right) \;\; \;\; A_1=\left( \begin{array}{cc}100 & 0 \\ 0 & 100 \end{array}\right)
\end{equation}
and
$$\phi(\sigma_1\sigma_2\dots\sigma_n)=\|A_{\sigma_1}A_{\sigma_2}\dots A_{\sigma_n}\| \;\; \;\, \sigma_i\in \{0, 1\}.$$
 Then $\phi$ is submultiplicative and nomalized the unit matrix, we have
$$\|\left( \begin{array}{cc} a & 0 \\ 0 & a \end{array}\right)\|=a$$
and hence for the periodic sequence
\begin{equation}
\|A_0A_1\dots A_0A_1\|=30^{2n}.
\end{equation} 
It follows that (1.7) is  $\Phi_n$ for every $n$ and $\Phi_*=30$.  

However, considered the set 
\begin{equation}
\|A_0A_1\dots  A_0A_1\dots A_0\dots\|=30^{2m}0.3^{n-2m} \;\; m=1,2,\dots
\end{equation} 
i.e. periodic product till $2m$ ended by all $A_0$,
we can findfrom this set $\Phi_m$  for every even $m$, however each infinit product of (1.8) tends to zero.
\end{remark}
\begin{theorem}  If we restrict to a closed subset $\mathcal{K}$  of $\Pi_{i=1}^\infty X_i$ then our  theorem  remains valid  if 
$\mathcal {K}$ is {\em shift-invariant} in the sense
$$\sigma_1\sigma_2\dots\sigma_n\dots \in\mathcal {K}\Rightarrow \sigma_2\sigma_3\dots\sigma_{n+1}\dots \in\mathcal {K}.$$
Thus, as we shall see by the next examples, our theorem will be  more applicable.
\end {theorem}
\begin {proof}
The proof is the same  as  Theorem 1.2,  with obvious modifications,.
\end{proof}
\begin{remark} 
 The subset of $\mathcal {K}$ with  $\phi_*(\sigma)=\Phi_*$ form a closed shift-invariant subspace.\\
\end{remark}
\begin {theorem}
 If $\sigma\in\mathcal {K}$ satisfies (1.3) then
$$\Phi_n^{\frac{1}{n}}\geq \phi(\sigma_1,\sigma_2\dots\sigma_n )^{\frac{1}{n}}\geq\Phi_* \;\; \;\;n=1,2,\dots.$$
and hence 
\begin{equation}
  \phi(\sigma_1,\sigma_2\dots\sigma_n )^{\frac{1}{n}}\rightarrow\Phi_*.
\end{equation}
It is easy to verify that also $ (1.2)\Leftrightarrow (1.6)$ holds.
\end{theorem}
\begin{proof} The theorem follows from the  submultiplicativity of sequence $\{\Phi_n \;\; n=1,2,\dots\}$.
\end{proof}

\section{Examples towards possible applications}

Each of the following example serves as a hint for possible application.
\begin{example}
Let  $\{X_i \;\; i=1,2,\dots\}$ be  bounded sets  $\{\mathcal {M}_i \;\;  i=1,2,\dots \}$ of square matrices with equal size, moreover,
$$\phi(\sigma_1\sigma_2\dots\sigma_k)=\|A_{\sigma_1}A_{\sigma_2}\dots A_{\sigma_n}\|$$
where $A_{\sigma_i}\in\mathcal {M}_{\sigma_i}$.  I.e.  a flow is represented by an infinite product of matrices.

In this setup we obtain the following\\
\\
{\em Proposition.} The infinite product
$$A_{\sigma_1}A_{\sigma_2}\dots  A_{\sigma_k}\dots  \;\; \;\; (A_{\sigma_k}\in\mathcal {M}_{\sigma_k})$$
tends to $0$ for every $\sigma$ if and only if there exists $N$ that
\begin{equation}
\|A_{\sigma_1}A_{\sigma_2}\dots A_{\sigma_N}\|<1
\end {equation}
 for every $\sigma$. \\

 In this case
$$\Phi_*= \lim_{n\rightarrow\infty} \|A_{\sigma_1}A_{\sigma_2}\dots A_{\sigma_n}\|^{\frac{1}{n}}$$
called  {\em the joint spectral radius} \cite{1}.
\end{example}
\begin{example}
 The finite set $\{F_1, F_2\dots F_N\}$  of functions of a metric space $X$ is called {\em joint contraction} if there is $M$ such that every composition 
\begin{equation} 
F_{\sigma_1}\circ F_{\sigma_2}\circ\dots\circ F_{\sigma_M} \;\; \;\; \sigma_i\in\{1,2,\dots ,N\}
\end{equation}
of length $M$ is a contraction.\\
{\em Proposition}  Let  $\{F_1, F_2\dots F_N\}$ be functions of a compact metric space $X$ with Lipschitz constant
$$\sup\frac{d[F_i(x)-F_i(y)]}{d[x-y]}= r_i \;\; \;\; i=1,2,\dots ,N.$$
Then the sequence  $\{F_{\sigma_1}\circ F_{\sigma_2}\circ\dots\circ F_{\sigma_n}(x) \;\; \;\; \sigma_i\in\{1,2,\dots ,N\}\}$ is convergent for every $x\in X$ and the limit is independent of $x$ if and only if $\{F_1, F_2,\dots , F_N\}$ is joint contraction.
\begin{proof}
 The "if" part can be proved in the same way as the Contraction Mapping Theorem.  The " only if " part follows from
 Theorem 1.2  In fact, the Lipschitz constant is submultiplicative.
 If $F$ and $G$ are Lipschitz function with Lipschitz constant $r_F$ and $r_G$ then 
$$\frac{ d[F\circ G(x), F\circ G(y)]}{d[x, y]}\leq \frac{ d[F\circ G(x), F\circ G(y)]}{ d[ G(x),  G(y)]} \;\; \frac{ d[G(x), G(y)]}{d[x, y]}\leq r_Fr_G.$$
It follows, that if $\phi(\sigma_1\sigma_2\dots\sigma_n)$ is the Lipschitz constant of  $F_{\sigma_1}\circ F_{\sigma_2}\circ\dots \circ F_{\sigma_n}$ then $\phi$ is submultiplicative and  hence there is a flow $\sigma$ that
$$\sup_{x,y}\frac{d[F_{\sigma_1}\circ F_{\sigma_2}\circ\dots\circ F_{\sigma_n}(x), F_{\sigma_1}\circ F_{\sigma_2}\circ\dots\circ F_{\sigma_n}(y)]}{d[x, y]}\geq \Phi_*^n$$
 from Theorem 1.2.

 If $\{F_1, F_2\dots F_N\}$ is NOT joint contraction, then  $\Phi_n\geq 1$ for every $n$ and hence $\Phi_*\geq 1$. It follows that there is a flow $\sigma^*$ such that the  Lipschitz constant  of  $F_{\sigma_1^*}\circ F_{\sigma_2^*}\circ\dots\circ F_{\sigma_n^*}$ is  equal or greater then $1$.  I.e. for every $n$ we have $x'_n , y'_n\in X$ that
$$\frac{d[F_{\sigma_1}\circ F_{\sigma_2}\circ\dots\circ F_{\sigma_n}(x'_n), F_{\sigma_1}\circ F_{\sigma_2}\circ\dots\circ F_{\sigma_n}(y'_n)]}{d[x'_n, y'_n]}\geq \Phi_*^n.$$
Consider the sequences $\{x'_n \;\; n=1,2,\dots \}$ and  $\{y'_n \;\; n=1,2,\dots \}$. Since $X$ is compact, we have convergent subsequences with limit $x'$ and $y'$ and after  a straightforward calculation we obtain also
$$\frac{d[F_{\sigma_1}\circ F_{\sigma_2}\circ\dots\circ F_{\sigma_n}(x'), F_{\sigma_1}\circ F_{\sigma_2}\circ\dots\circ F_{\sigma_n}(y')]}{d[x', y']}\geq \Phi_*^n.$$
It follows that the convergence of 
$$\{F_{\sigma_1}\circ F_{\sigma_2}\circ\dots\circ F_{\sigma_n}(x) ; \;\; n=1,2,\dots\}$$
depends on $x$, moreover, the squence tend to infinity in an exponential rate $\Phi_*^n$ if $\Phi_*>1$.
\end{proof}  
\begin {remark}
 If  $\{F_1, F_2,\dots , F_N\}$ is NOT  joint contraction, then often we can choose a subset $\mathcal {K}$  that  the 
 sequences  $\{F_{\sigma_1}\circ F_{\sigma_2}\circ\dots\circ F_{\sigma_n}(x) \;\; \;\;n=1,2,\dots \}$, corresponded to
 $\mathcal {K}$, are convergent with limit independent of $x$. 

The set $\mathcal {K}$ is obtained by to select words $\sigma_i\sigma_{i+1}\dots\sigma_j$ and considered those sequences  which do not contain $F_{\sigma_i}\circ F_{\sigma_{i+1}}\circ\dots\circ F_{\sigma_j}$.  Thus for the infinite compositions belonging to  $\mathcal {K}$ there is $M$ that every composition (2.2) is a contraction.
 \end{remark}
\end{example}
\begin{example}
\cite {3} Let $\sigma$ be the input of a program $P$ of a Turing machine and $s_n(\sigma)$ be the number of {\em different} cells visited during $n$ step of the execution of $\sigma$.   We are interested in the input $\sigma$ that takes the average
$$\frac{s_n(\sigma)}{n}$$
maximal after long time. There is subadditivity
 $$s_{k+m}(\sigma)\leq s_k(\sigma)+s_m(P^k(\sigma))$$ 
where  $P^k(\sigma)$ is the output after  the $k$-th step of the execution of  $\sigma$ ($\sigma$ is a string  with finite alphabet).  Hence, applied  Theorem 1.2  when   $\phi(\sigma_1\sigma_2\dots\sigma_n)=s_n(\sigma)$,  there exists  input  $\sigma$ with
 $$\frac{s_n(\sigma)}{n}\geq \Phi_*  \;\; \;\; n=1,2,\dots$$
and this is also the desired maximum.
\end{example}
\begin{example}
\cite {5} Let $(X, d)$ be a compact metric space and $T: \; X\Rightarrow X$ a continuous map. For every open cover $\mathcal {U}$ of $X$  there is a finite subcover $\mathcal {V}$, since $X$ is compact. Define $N(\mathcal {U})$ as the minimal number of a subcover.

If $\mathcal{U}, \mathcal  {V}$ are open covers, then
$$\mathcal{U}\vee\mathcal{ V}=\{u_i\bigcap v_j ; \; u_i\in\mathcal{U},  v_j\in\mathcal {V} \}.$$
It easy to check that $N(\bullet)$ is subadditive in the following sense
$$X(\mathcal{U}\vee\mathcal {V})\leq N(\mathcal{U})+N(\mathcal {V}).$$
The {\em topological entropy} $h(T \mathcal{U})$  is defined by these concepts as follows
$$h(T, \mathcal{U})=\lim_{n\rightarrow\infty}\frac{1}{n}\log N(\mathcal{U}_0^{n-1})$$
where 
$$ N(\mathcal{U}_0^{n-1})= \vee_{k=1}^{n-1}T^{[-k]}\mathcal{U}.$$
The limit exists because of the subadditivity of $N(\bullet)$ with $a_n=\log N(\mathcal{U}_0^{n-1})$.\\

The {\em topological entropy $h(T)$ of $T$} is the supremum of $\{h(T, \mathcal{U}) \;\; \mathcal{U}$ is an open cover of $X \}$.

If there is a compact topology of the minimal covers $\mathcal {V}$ induced by the open sets of $X$ such that 
$$\mathcal {V}\Longrightarrow N(\mathcal{U})$$
is continuous, then there is a cover $\mathcal{U}$ of the compact space $X$ that $h(T)=h(T, \mathcal{U})$.
\end{example}

\bibliographystyle{amsplain}

 \curraddr{Institute of Mathematics, Budapest University of Technology and Economics}

\email{\it E-mail address: mate@math.bme.hu}

\end{document}